\documentclass{ws-book9x6}
\usepackage{ws-book-thm}
\usepackage{ws-book-van}   % Numbered citations
\usepackage{amsmath}
\usepackage{amssymb}
\usepackage{multicol}
\usepackage{multirow}
\usepackage{cancel}
\usepackage{cases}
\usepackage{lastpage}
\usepackage{fancyhdr}
\usepackage{pst-all}
\usepackage[utf8]{inputenc}
\usepackage[center]{caption}
\usepackage{subfigure}

\usepackage{hyperref}% permet de cr\`{e}er des hyperliens dans le document PDF
\usepackage{amsfonts}% utilise certaines polices de l'American mathematical society
\usepackage{eurosym}% permet d'\`{e}crire un symbole de l'euro
\usepackage{listings}% permet d'inclure des codes sources (programmes)
\usepackage{dsfont}
\usepackage{graphicx}%pour inclure des graphiques/images

\usepackage[T1]{fontenc}% usage des caractËres accentu\`{e}s
\usepackage{dsfont}% permet d'utiliser la fonction \mathbb avec des chiffres
\usepackage{latexsym,stmaryrd}%permet l'utilisation de la police des symboles latex
\usepackage{verbatim,enumerate}% verbatim : permet d'utiliser l'environnement verbatim
\usepackage[perpage]{footmisc} %pour recommencer ‡ indicer les notes de bas de pages ‡ 1 ‡ chaque page
\usepackage{tabularx}

\usepackage{commath} % for \dif.
\usepackage{relsize}
\usepackage{algorithm, algpseudocode,amsmath,amssymb,bm,mathtools,mathrsfs,multirow,url} % tikz,
\usepackage{upgreek} % for pi
\usepackage{array}
\usepackage{framed} % to make boxes

\numberwithin{equation}{section}
\fancypagestyle{empty}{%
\fancyhf{} % clear all header and footer fields
}

\fancypagestyle{plain}{%
\fancyhf{} % clear all header and footer fields
\fancyfoot[LE,RO]{\thepage}

}

\fancypagestyle{headings}{%
\fancyhf{} % clear all header and footer fields
\fancyfoot[LE,RO]{\thepage}

}

\fancypagestyle{monstyle2}{%
\fancyhf{} % clear all header and footer fields
\fancyhead[LE,RO]{\thepage}
}

\fancypagestyle{monstyle}{
\fancyhf{} % clear all header and footer fields
\fancyfoot[LE,RO]{\thepage}
\fancyhead[LE,RO]{\textit{\leftmark}}
}

\def \NZ{\mathbb{N}_0}

\newcommand{\RL}{\mathbb{R}}

\newcommand{\Laplace}{\mathscr{L}}
\newcommand{\e}{\mathrm{e}}
\newcommand{\ve}{\bm{\mathrm{e}}} % vector e

\renewcommand{\L}{\mathcal{L}} % e.g. L^2 loss.

\newcommand{\Oh}{{\mathcal{O}}}
\newcommand{\Exp}{\mathbb{E}}

\newcommand{\Norm}{\mathcal{N}}
\newcommand{\LN}{\mathcal{LN}}
\newcommand{\SLN}{\mathcal{SLN}}

\newcommand\bfSigma{\bm{\Sigma}}

\newcommand{\bfH}{\bm{H}}
\newcommand{\bfI}{\bm{I}}

\newcommand{\bfS}{\bm{S}}

\newcommand{\bfX}{\bm{X}}

\newcommand{\bfx}{\bm{x}}
\newcommand{\bfy}{\bm{y}}
\newcommand{\bfz}{\bm{z}}

\newcommand\bfzero{\bm{0}}
\newcommand\bfone{\bm{1}}

\newcommand\bfmu{\bm{\mu}}

\newcommand{\diag}{\mathrm{diag}}

\newcommand{\Var}{\mathbb{V}\mathrm{ar}}

\renewcommand{\epsilon}{\varepsilon}

\renewcommand{\pi}{\uppi}

\newcommand{\tr}{\top}

\makeatletter
\newcommand{\oset}[3][0ex]{
  \mathrel{\mathop{#3}\limits^{
    \vbox to#1{\kern-2\ex@
    \hbox{$\scriptstyle#2$}\vss}}}}
\makeatother

\newcommand{\iidDist}{\overset{\mathrm{i.i.d.}}{\sim}}

\newcommand{\stimes}{{\mkern-2mu\times\mkern-2mu}}

\graphicspath{{images/}}

\newenvironment{abstract}
                 {\vspace{6pt}
                  \begin{center}
                  \begin{minipage}{315.83pt}
                  \small
                  \begin{center}
                  {\bf Abstract}
                  \end{center}
                  \noindent\ignorespaces
                 }
                 {\end{minipage}\end{center}}

\title{Norberg Festschrift}

\begin{document}

% The "{ }" below is so the running title looks correct.
\chapter{Orthonormal polynomial expansions \\{ }and lognormal sum densities}
By S{\o}ren Asmussen, Pierre-Olivier Goffard, Patrick J.\ Laub

\begin{abstract}
Approximations for an unknown density $g$ in terms of a reference density $f_\nu$
and its associated orthonormal polynomials are discussed. The main application is
the approximation of the density $f$ of a sum $S$ of lognormals which may have different
variances or be dependent. In this setting, $g$ may be $f$ itself or a transformed
density, in particular that of $\log S$ or an exponentially tilted density.  Choices
of reference densities $f_\nu$ that are considered include normal, gamma and lognormal densities.
For the lognormal case, the orthonormal polynomials are found in closed form
and it is shown that they are not dense in $L_2(f_\nu)$, a result that is closely related
to the lognormal distribution not being determined by its moments and provides a warning to the most
obvious choice of taking $f_\nu$ as lognormal. Numerical examples are presented
and comparison are made to  established approaches such as the Fenton--Wilkinson method
and skew-normal approximations. Also extension to density estimation for statistical data sets and non-Gaussian copulas
are outlined.\\
\textit{Keywords:} Lognormal distribution, sums of lognormally distributed random variable, orthogonal polynomial,
density estimation, Stieltjes moment problem, numerical approximation of functions, exponential tilting,
conditional Monte Carlo, Archimedean copula, Gram--Charlier expansion, Hermite polynomial,
Laguerre polynomial
\end{abstract}

% ORDER ALPHABETICALLY
% numerical inversion of Laplace transform,

%%%%%%%%%%%%%%%%%%%%%%%%%%%%%%%%%%   The new notation.  %%%%%%%%%%%%%%%%%%%%%%%%%%%%%%%%%%%%%%%%%
% \begin{itemize}
% 	\item $S := \log(\e^{X_1} + \dots + \e^{X_n}) \sim \SLN(\bfmu, \bfSigma)$ is the sum of $n$ lognormals.
% 	\item $Z := \log(S)$ is the log of a $\SLN$ random variable (used in the Laguerre expansion).
% 	\item $f(x)$ is the p.d.f.\ of $S$.
% 	\item $f_{\Norm(\mu,\sigma^2)}(x)$ is the $\Norm(\mu, \sigma^2)$ p.d.f.\ and $f_{\LN(\mu,\sigma^2)}(x)$ is the $\LN(\mu, \sigma^2)$ p.d.f.
% 	\item $\widehat{f}(x) = \sum_{k=1}^K a_k Q_k(x) f_\nu(x)$ is the polynomial approximation to $f(\cdot)$ (instead of $f_X^K(x)$ which I think looks a bit clunky).
% 	\item $f_\theta(x)$ is the exponentially tilted $\SLN$ ($\SLN_\theta$) p.d.f.
% 	\item $S_1$, \dots, $S_R \iidDist \SLN(\bfmu, \bfSigma)$ are i.i.d.\ replications of $\SLN$ random variables.
% 	\item $S_1^\theta$, \dots, $S_R^\theta \iidDist \SLN_\theta(\bfmu, \bfSigma)$ are i.i.d.\ replications of $\SLN_\theta$ random variables.
% 	\item $\Laplace(\theta)$ is the Laplace transform of $\SLN$, and $\Laplace_i$ are the ``derivatives''.
%	\item Gamma distribution is Gamme$(r,m)$??
% \end{itemize}

%!TEX root = ../main.tex
\section{Introduction}\label{S:Intr}

The lognormal distribution arises in a wide variety of disciplines such
as engineering, economics, insurance, finance, and
across the sciences \cite{AitchisonBrown1957,CrowShimizu1988,
JohnsonKotzBalakrishnan94,LimpertStahelAbbt2001,Dufresne2009}.
Therefore, it is natural that sums $S$ of $n$ lognormals come up in a
number of contexts. A basic example in finance is the
Black--Scholes
model, which assumes that security prices are % can be modeled as %independent
lognormals, and hence the value of a
portfolio with $n$  securities has the form $S$.
%Another example occurs in the valuation of arithmetic Asian options
%where the payoff depends on the finite sum of correlated lognormals
%\citep{MilevskyPosner1998,Dufresne2004}.
In  insurance, individual claim sizes are often taken as independent
lognormals, so the total claim amount in a certain period is
again of form \cite{ThorinWikstad1977}.
A further example occurs
in telecommunications, where
the inverse of the signal-to-noise ratio (a measure of performance in
wireless systems)
can be modeled as a sum of i.i.d.\ lognormals \citep{Gubner2006}.

The distribution $S$ is, however, not available
in explicit form,
and evaluating it numerically or approximating it is considered to be a challenging problem
with a long history.
The classical approach is to use an approximation
with another lognormal distribution.  This goes back at least to
\cite{Fenton1960} and it is nowadays
known as the \emph{Fenton--Wilkinson method} as
according to \cite{Marlow1967} this approximation was already
used by Wilkinson in 1934.
However, it can be rather inaccurate when the number of summands is rather small, or
when the dispersion
parameter is too high.
Also tail approximations have been extensively discussed, with the
right tail being a classical example in subexponential theory,\cite{EKM97}, and the
study of the left tail being more recent, \cite{AsJeRo14b}, \cite{GuTa13}.
% Numerical inversion of the Laplace transform \framebox{TBC}.

This paper discusses a different method, to approximate the probability density function (p.d.f.) $f$ via polynomials $\{Q_k\}$ which are orthonormal w.r.t.\ some reference
measure $\nu$. In the general formulation, one is interested in approximating a target density $g$ using the density $f_\nu$ of $\nu$
as reference and $g$ some other density. One then finds
a series representation of $g/f_\nu$ of the form $\sum_{k=0}^\infty a_k Q_k$,
and then the approximation of $g$ is
\begin{equation} \label{orthog_approx}
	\widehat{g}(x)\ =\ g_\nu(x) \sum_{k=0}^K a_k Q_k(x),
\end{equation}
for some suitable $K$.
%A more detailed account of this approach is given
%in Section~\ref{S:OrthPol}.
The most obvious connection to the lognormal sum problem is $g=f$, but we shall look
also at other possibilities, to take $g$ as the density of $\log S$ and transform back to get the approximation $\widehat{f}(x)=$
$\widehat{g}(\log x)/x$ or to use an exponential tilting.
The choice of $\nu$ is a crucial step, and
three candidates for $\nu$ are investigated: the normal, the gamma, and the lognormal distributions.

The form of the $Q_k$ is classical for the normal distribution where it is the Hermite polynomials and for the gamma where  it is the Laguerre polynomials,
but for the lognormal distributions it does not appear to be in the literature and we give here  the functional expression
(Theorem~\ref{LogNormalOrthogonalPolynomialProposition}). The Fenton--Wilkinson method may be seen
as the $K=2$ case of $f_\nu$ being lognormal of the general scheme, and this choice of $f_\nu$  may be the most obvious one.
However, we show that in the lognormal case the orthonormal polynomials are not dense in $L_2(f_\nu)$. This result is closely related
to the lognormal distribution not being determined by its moments
\cite{Heyde63,Berg1984} and indicates that a lognormal $f_\nu$ is potentially dangerous.
For this reason, the rest of the paper concentrates on  taking the reference distribution as normal (using the logarithmic transformation)
or gamma  (using exponential tilting).

After discussing the details of the orthonormal polynomials expansions in Sections~\ref{S:OrthPol} and~\ref{S:LN_Appl},
we proceed in Section~\ref{S:numerical} to show a number of numerical examples. The polynomial expansions
are compared to existing methods as Fenton--Wilkinson and a more recent approximation in terms of log skew normal
distributions \cite{hcine2015highly}, as well as to exact values obtained by numerical quadrature in cases where this
is possible or by Monte Carlo density estimation. Section~\ref{S:numerical} also outlines an extension to  statistical data sets and non-Gaussian copulas. Appendices A.1 contains a technical proof
and Appendix A.2 some new material on the $\SLN$ Laplace transform.

%!TEX root = ../main.tex
\section{Orthogonal polynomial representation of probability density functions}
\label{S:OrthPol}

Let $X$ be a random variable which has a density $f$
with respect to some measure $\lambda\ge 0$
(typically Lebesgue measure on an interval or counting measure
on a subset of $\mathbb{Z}$).  If $f$ is unknown but the distribution of $X$ is
expected to be close to some probability measure $\nu$ with p.d.f.\ $f_{\nu}$,
one may use $f_{\nu}$ as a first approximation to $f$ and next try
to improve by invoking suitable correction terms.

In the setting of this paper $X$ is the sum of lognormal r.v.s and the correction terms are obtained by expansions in terms of orthonormal polynomials.
Before going into the details of the lognormal example, let us consider the general case.

Assuming all moments of $\nu$ to be finite,
the standard Gram--Schmidt orthogonalization technique shows the existence
of a set of polynomials $\{Q_k\}_{k\in\NZ}$ which are orthonormal in
$\L^{2}(\nu)$ equipped with the usual inner product
$\left<g,h\right>=\int gh \dif \nu$ and the corresponding norm
$\|g\|^2=\left<g,g\right>$. That is,
the $Q_k$ satisfy
\begin{equation}\label{OrthogonalityCondition}
\left<Q_i,Q_j\right>=\int Q_i(x)Q_j(x) \dif \nu(x)=\delta_{ij} \,, \quad i, j \in\NZ,
\end{equation}
where $\delta_{ij}$ denotes the Kronecker symbol.
If there exists an $\alpha>0$ such that
\begin{equation}\label{DefinedMGFCondition}
\int \e^{\alpha|x|} \dif \nu(x) < \infty\,,
\end{equation}
the set  $\{Q_k\}_{k\in\NZ}$ is complete in $\L^{2}(\nu)$, cf.\ Chapter $7$ of the book by Nagy \cite{Na65}. The implication is that if $f/f_\nu$ is in $\L^{2}(\nu)$, that is, if
\begin{equation}\label{L2Condition1}
\int \frac{f(x)^2}{f_\nu(x)^2}\dif \nu(x)\ =\ \int \frac{f(x)^2}{f_\nu(x)}\dif \lambda(x)\ <\ \infty\,,
\end{equation}
we may expand $f/f_\nu$ as $\sum_{k=0}^\infty a_k Q_k$ where
\begin{align}\label{28.10a}a_k\ &=
\left<f/f_\nu,Q_k\right>\ =\ \int f Q_k\,\dif\lambda\ =\ \Exp\left[Q_k(X)\right]\,.
\end{align}
This suggests that we use \eqref{orthog_approx} as an approximation of $f$ in situations where the p.d.f.\ of $X$ is unknown but the moments are accessible.

\begin{remark}\label{Rem:27.10a}
If the first $m$ moments of $X$ and $\nu$ coincide, one has $a_k=0$ for $k=1$, \ldots, $m$.
When choosing $\nu$, a possible guideline is therefore to match as many moments as possible.
\hfill $\Diamond$
\end{remark}

Due to the Parseval relationship $\sum_{k=0}^\infty a_k^2=\|f/f_\nu\|^2$, the coefficients of the polynomial expansion, $\{a_k\}_{k\in\NZ}$, tend toward $0$ as $k \to \infty$. The accuracy of the approximation \eqref{orthog_approx}, for a given order of truncation $K$, depends upon how swiftly the coefficients decay; note that the $\L^{2}$ loss of the approximation of $f/f_\nu$ is $\sum_{K+1}^{\infty}a_k^2$.
%The computation of the coefficients of the expansion is widely addressed in this work.
Note also that the orthogonal polynomials can be specified recursively
(see Thm.~3.2.1 of \cite{Sz39}) which allows a reduction of the computing time required for the coefficients' evaluation
and makes it feasible to consider rather large $K$.
%The choice of the reference probability measure (and hence the corresponding orthogonal polynomials) has a great impact on the performance of the polynomial approximations and is discussed next.

\subsection{Normal reference distribution}\label{SS:NormalNu}
A common choice as a reference distribution is the normal $\mathcal{N}(\mu,\sigma^{2})$. The associated orthonormal polynomial are given by
\begin{equation}\label{eq:NormalDistributionOrthogonalPolynomial}
Q_{k}(x)=\frac{1}{k!2^{k/2}}H_{k}\left(\frac{x-\mu}{\sigma\sqrt{2}}\right),
\end{equation}
where $\left\{H_{k}\right\}_{k\in\NZ}$ are the Hermite polynomials, defined in \cite{Sz39} for instance.
If $f$ is continuous, a sufficient (and close to necessary) condition for $f/f_\nu\in\L^2(\nu)$ is
\begin{equation}\label{SA7.11a}
f(x)\ =\ \Oh\big(\e^{-ax^2}\big)\quad\text{as }x\to\pm\infty\quad\text{with } a>\big(4\sigma^2\big)^{-1}\,.
\end{equation}
Indeed, we can  write the integral in \eqref{L2Condition1} as $I_1+I_2+I_3$, the integrals over
$(-\infty,-A)$, $[-A,A]$, resp.\ $(A,\infty)$. Note that $I_2<\infty$  follows since the integrand
is finite by continuity, whereas the finiteness of $I_1,I_3$ is ensured by the integrands
being $\Oh(\e^{-bx^2})$ where $b=2a-1/2\sigma^2>0$. Similar arguments apply to
conditions \eqref{SA7.11b} and \eqref{SA7.11c} below.
\begin{remark}
The expansion formed by a standard normal baseline distribution and Hermite polynomials is known in the literature as Gram--Charlier expansion of type A, and the application to a standardised sum is the Edgeworth expansion, cf.\ \cite{Cr99}, \cite{Barndorff1989asymptotic}. \hfill $\Diamond$
\end{remark}
%\begin{example}\label{Ex:27.10a}
%The classical example of this technique (though not usually put in this framework)
%is the Edgeworth expansion, also associated with the names of Charlier and Gram.
%Consider $Y_1$, \dots $Y_n$ i.i.d.\ r.v.s where $\Exp[Y_i]=\mu$ and $\Var[Y_i]=\omega^2$.
%The Edgeworth expansion covers $X=(S-n\mu)/(\sqrt{n}\,\omega)$ with $\Norm(0,1)$ as the reference measure $\nu$.
%
%The Edgeworth example illustrates well several aspects of the theory. In the scenario of Remark~\ref{Rem:27.10a} with $m=2$, condition~\eqref{DefinedMGFCondition} is trivially satisfied,
%but condition~\eqref{L2Condition1} requires more attention. It actually fails unless
%$f(x)$ decays very quickly as $x \to\pm \infty$, even in such a basic example as the $Y_i$
%being standard exponential (then $f(x)$ is of order $x^{n-1} \e^{-xn^{1/2}}$
%which multiplied by $\e^{x^2/2}$ does not integrate). Nevertheless, the Edgeworth approach
%has been observed to perform well even when condition~\eqref{L2Condition1} is not satisfied;
%in fact, the condition is only sufficient, not necessary (Cram\'er showed it can be relaxed to
%$\Exp \, \e^{Y^2/4\omega}<\infty$, \cite[p.\ ??]{Barndorff1989asymptotic}).\hfill $\Diamond$
%\end{example}

\subsection{Gamma reference distribution}\label{SS:GammaNu}
If $X$ has support $(0,\infty)$, it is natural to look for a $\nu$ with the same property.
One of the most apparent possibilities is the
gamma distribution, denoted Gamma$(r,m)$ where $r$ is the shape parameter
and $m$ the scale parameter.
The p.d.f.\ is
\begin{equation} \label{GammaPDF}
f_{\nu}(x)=\frac{x^{r-1} \e^{-x/m}}{m^{r}\Gamma(r)},\hspace{0.2cm}x\in \RL_+ \,.
\end{equation}
The associated polynomials are given by
\begin{equation} \label{OrthogonalPolynomialGammaDistribution}
Q_{n}(x)=(-1)^{n}\left[\frac{\Gamma(n+r)}{\Gamma(n+1)\Gamma(r)}\right]^{-1/2}L_{n}^{r-1}(x/m),\hspace{0.2cm} n\in\NZ,
\end{equation}
%\begin{equation} \label{eq:GeneralizedLaguerrePolynomialExpression}
%L^{r-1}_{n}(x)=\sum_{i=0}^{n}\binom{n+r-1}{n-i}\frac{(-x)^{i}}{i!},\hspace{0.2cm}n\in\mathbb{N},
%\end{equation}
where $\{L^{r-1}_{n}\}_{n\in\NZ}$ denote the \emph{generalised Laguerre polynomials}, see \cite{Sz39}; in Mathematica these are accessible via the \texttt{LaguerreL} function.
%Note, the binomial coefficients above are to be interpreted as the relevant gamma functions for $r \not\in \N$.
%The polynomials defined in \eqref{OrthogonalPolynomialGammaDistribution} satisfy the recurrence relationship
%\begin{eqnarray} \label{OrthogonalPolynomialGammaDistributionRecurreceRelationship}
%nQ_{n}(x)&=&\left(\frac{x}{m}-2n-r+2\right)Q_{n-1}(x)\sqrt{\frac{n}{n+r-1}}\nonumber\\
%&-&(n+r-2)Q_{n-2}(x)\sqrt{\frac{n(n-1)}{(n+r-1)(n+r-2)}}.
%\end{eqnarray}
%The recurrence relationship will be employed later to speed up the computation of the coefficients.
Similarly to \eqref{SA7.11a}, one has the following condition for $f/f_\nu\in\L^2(\nu)$:
\begin{gather}\begin{split}\label{SA7.11b}
f(x)\ &=\ \Oh\big(\e^{-\delta x}\big)\quad\text{as }x\to\infty\quad\text{with } \delta>1/2m\,, \text{ and }\\
f(x)\ &=\ \Oh\big(x^\beta\big)\quad\text{as }x\to 0\quad\text{with }\beta>r/2-1\,.
\end{split}\end{gather}

% For the $\L^2$ condition~\eqref{L2Condition1} to hold, both the left and the right tail of $X$ plays
% a role. A sufficient condition is that  $f(x)=\Oh(\e^{-\delta x})$ as $x\to\infty$
% for some $\delta>1/2m$ and that $f(x)=\Oh(x^\beta)$ as $x\to 0$
% for some $\beta>r/2-1$.

%\begin{remark}
%The use of gamma distribution and Laguerre polynomials links our approach to a well established technique called the Laguerre method. The expansion is an orthogonal projection onto the basis of Laguerre functions constructed by multiplying Laguerre polynomials and the square root of the exponential distribution with parameter $1$. The method is described in \cite{AbChWh95}, note also that the damping procedure employed when integrability problems arise is quiet similar to considering the exponentially tilted distribution instead of the real one. $\hfill \Diamond$
%\end{remark}

\subsection{Lognormal reference distribution}\label{SS:LogNNu}
The lognormal distribution $\LN(\mu, \sigma^2)$ is the distribution of $\e^Y$ where $Y \sim \Norm(\mu, \sigma^2)$. It has support on $\RL_+$. The polynomials orthogonal to the $\LN(\mu, \sigma^2)$ are given in the following proposition, to be proved in the Appendix:

\begin{theorem}\label{LogNormalOrthogonalPolynomialProposition}
The polynomials orthonormal with respect to the lognormal distribution are given by
\begin{equation}\label{eq:OrthonormalPolynomialLogNormalExpression}
Q_{k}(x)=\frac{\e^{-\frac{k^{2}\sigma^{2}}{2}}}{\sqrt{\left[\e^{-\sigma^{2}},\e^{-\sigma^{2}}\right]_{k}}}\sum_{i=0}^{k}(-1)^{k+i}\e^{-i\mu-\frac{i^{2}\sigma^{2}}{2}}e_{k-i}\left(1,\ldots,\e^{(k-1)\sigma^{2}}\right)x^{i},
\end{equation}
for $k\in\NZ$ where
\begin{equation}\label{eq:ElementarySymmetricPolynomial}
e_{i}\left(X_{1},\ldots,X_{k}\right)=
\begin{cases}
\sum_{1\leq j_{1}<\ldots<j_{i}\leq k}X_{j_{1}}\ldots X_{j_{i}}, &\mbox{for } i\leq k, \\
0,& \mbox{for } i>k,
\end{cases}
\end{equation}
are the elementary symmetric polynomials and $\left[x,q\right]_{n}=\prod_{i=0}^{n-1}\left(1-xq^{i}\right)$ is the Q-Pochhammer symbol.
\end{theorem}
%\begin{proof}
%See the Appendix.
%\end{proof}

\begin{remark}\label{Rem:27.10b}
The result of Theorem~\ref{LogNormalOrthogonalPolynomialProposition}
does not appear to be in the literature; the closest reference
seems to be a 1923 paper by Wigert~\cite{Wi23} who considers
the distribution with p.d.f.\
%$f_{\mathcal{SW}}(x)=
$\ell\e^{-\ell^{2}\ln^{2}(x)}/\sqrt{\pi}$
(later called the Stieltjes--Wigert distribution).
\hfill $\Diamond$ \end{remark}

The equivalent of condition \eqref{SA7.11a} for $f/f_\nu\in\L^2(\nu)$ now becomes
\begin{align}\label{SA7.11c}
f(x)\ &=\ \Oh\big(\e^{-b \log^2x}\big)\quad\text{for }x\to 0 \text{ and } \infty \quad\text{with } b> \big(4\sigma^2\big)^{-1}\,,
\end{align}
which is rather mild. However,
a key difficulty in taking the reference distribution as lognormal is the following
result related to the fact that the lognormal and the Stieltjes-Wigert distributions are not characterised by their moments, see \cite{Heyde63,Berg1984,Ch79,Ch03}. Hence, the orthogonal polynomials associated with the lognormal p.d.f.\ and the Stieltjes-Wigert p.d.f.\ are also the orthogonal polynomials for some other distribution.
\begin{proposition} \label{prop:ln_incomplete} The set of orthonormal polynomials in
Theorem~\ref{LogNormalOrthogonalPolynomialProposition} is incomplete
in $\L^2(\nu)$. That is, {\rm span}$\{Q_k\}_{k \in \NZ}$ is a proper subset of $\L^2(\nu)$.
\end{proposition}
\begin{proof} Let $Y$ be a r.v.\ whose distribution is the given lognormal $\nu$
and $X$ a r.v.\ with a distribution different from $Y$
but with the same moments. According to \cite[pp.\ 201--202]{Berg1984} such an $X$ can be
chosen such that $f_X/f_\nu$ is bounded and hence in $\L^2(\nu)$. The projection of
$f/f_\nu$  onto ${\rm span}\{Q_k\}$ is then
\begin{align*}\sum_{k=0}^\infty \left<f/f_\nu,Q_k\right>Q_k\ &=\
\sum_{k=0}^\infty \Exp\left[Q_k(X)\right]Q_k\ =\ \sum_{k=0}^\infty \Exp\left[Q_k(Y)\right]Q_k\\
&=\ Q_0\ =\ 1\,\neq f/f_\nu,
\end{align*}
where the first step used \eqref{28.10a} and the second that the moments are the same.
This implies $f/f_\nu\in\L^2(\nu)\setminus{\rm span}\{Q_k\}$ and the assertion.
\end{proof}
%
%The problem of the incompleteness of the orthogonal polynomial system $\{Q_k\}_{k\in\NZ}$ in $\L^{2}(\nu)$ already appeared in  Wigert \cite{Wi23} where it is stated in the introduction that these polynomials cannot be used to approximate any continuous and bounded function.

\subsection{Convergence of the estimators  w.r.t.\ $K$ }

Orthogonal polynomial approximations generally become more accurate as the order of the approximation $K$ increases. Figure~\ref{fig:converg} shows a specific orthogonal polynomial approximation, $\widehat{f}_{\Norm}$ (to be described in Section~\ref{SS:LNNormalNu}), converging to the true $\SLN$ density $f$ for increasing $K$. In this example, we take the $\SLN$ distribution with $\mu=(0,0,0)^\tr$, $\Sigma_{ii} = 0.1$, and $\rho = -0.1$.

\begin{figure}
\centering
\includegraphics[width=\textwidth]{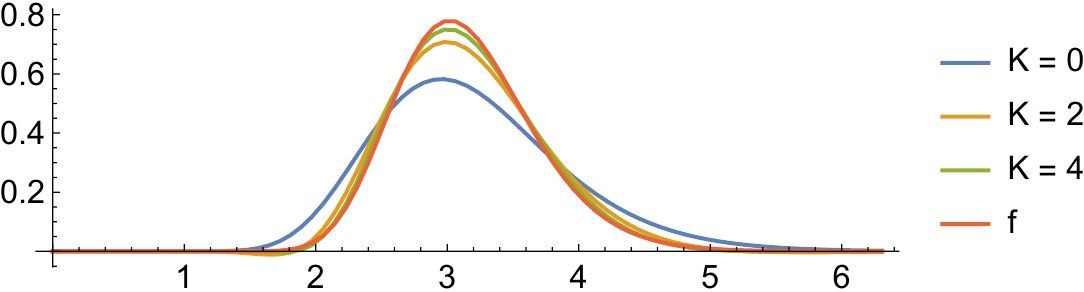}
\caption{Examples of orthogonal polynomial approximations using a $\Norm(1.13,0.23^2)$ reference converging to the target $f$ with increasing $K$.}
\label{fig:converg}
\end{figure}

Proposition~\ref{prop:ln_incomplete} implies that orthogonal polynomial approximations with a lognormal reference distribution cannot be relied upon to converge to the desired target density but may have a different limit (the orthogonal projection described there). The next plot, Figure~\ref{fig:bad_converge}, illustrates this phenomenon. The approximation appears to converge, but not to the target density.
Our theoretical discussion suggests that this incorrect limit density has the same moments as the target lognormal distribution,
and this was verified numerically for the first few moments,
%Interestingly (or perhaps unsurprisingly) this incorrect limit density has the same moment sequence as the target lognormal distribution.

\begin{figure}
\centering
\includegraphics[width=\textwidth]{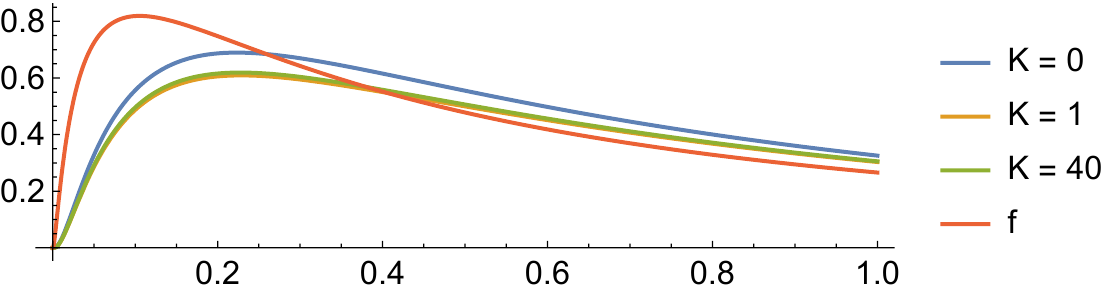}

\caption{Example of orthogonal polynomial approximations of $f$ using a $\LN(0, 1.22^2)$ reference not converging to the $\LN(0, 1.50^2)$ target.}
\label{fig:bad_converge}
\end{figure}

Lastly, it must be noted that we cannot in practice take $K$ arbitrarily large, due to numerical errors incurred in calculating the $\{a_k\}$ coefficients. Obviously this can be overcome by using infinite precision operations, however this swiftly becomes prohibitively slow. Software tools like \texttt{Mathematica} allow for arbitrarily large but finite precision, which gives on the flexibility to choose a desired accuracy/speed trade-off. We use this technology and select $K \le 40$.

%!TEX root = ../main.tex
\section{Application to  lognormal sums}\label{S:LN_Appl}
We now turn to our main case of interest where $X=S$ is a lognormal sum. Specifically,
\begin{equation}\label{7.12a}
S=\e^{X_{1}}+\ldots+\e^{X_n} \,,\quad n \ge 2\,,
\end{equation}
where the vector $\bfX=\left(X_{1},\ldots,X_{n}\right)$ is governed by a multivariate normal distribution $\Norm\left(\bfmu, \bfSigma\right)$, where $\bfmu=(\mu_{1},\ldots,\mu_{n})^\tr$ is the mean vector and $\bfSigma=\left(\sigma_{ij}\right)$ the covariance matrix.
We write this distribution as $\SLN\left(\bfmu, \bfSigma \right)$, and hereafter denote its p.d.f.\ as $f$. We are interested in computing the p.d.f.\ when the summands exhibit dependency ($\bfSigma$ is non-diagonal). This is an ambitious goal given that the p.d.f.\ of the sum of two i.i.d lognormally distributed random variables is already unknown. The validity of the polynomial approximations rely on the $\L^2$ integrability condition~\eqref{L2Condition1}, which is difficult to check because the p.d.f.\ of $S$ is not available. We will need asymptotic results describing the left and the right tail of the distribution of $S$, which we collect in the following subsection.

\subsection{Tail asymptotics of lognormal sums}\label{SS:LNSumsAsympt}

The tail asymptotics of $f(x)$ are given in the following lemma, which simply collects the results from Corollary~2 of \cite{tankov2015tail} and Theorem~1 of \cite{AsRo08}.

\begin{lemma} \label{le:SLNTails}
We have
\begin{eqnarray}
	f(x) &=&  \Oh ( \exp\{-c_1 \ln(x)^2 \} ) \text{ as }x\to 0 \text{ and } \label{SLNLeftTail} \\
	f(x) &=& \Oh ( \exp\{-c_2 \ln(x)^2 \} ) \text{ as }x\to \infty \label{SLNRightTail}
\end{eqnarray}
where
\[
	c_1 = \big[ 2 \min_{\bm{w} \in \Delta} \bm{w}^\tr \bm{\Sigma}^{-1} \bm{w} \big]^{-1} \, \text{ and } \,\,\, c_2 = \big[ 2 \max_{i=1,\dots,n} \sigma_{ii} \big]^{-1}  \,,
\]
with the notation that $\Delta = \{ \bm{w} \,|\, w_i \in \RL_+, \sum_{i=1}^n w_i = 1 \}$.
\end{lemma}

We are also interested in the asymptotic behaviour of $Z = \ln(S)$ later in the paper. Writing the p.d.f.\ of $Z$ as $f_Z$ we have $f_Z(z) = \e^{z}f(\e^{z})$. Together with L'H\^opital's rule this gives the following results (extending \cite{GaXuYe08}):

\begin{corollary} \label{co:LogSLNTails}
We have
\begin{align} \label{LogSLNLeftTail}
	f_Z(z)  \ &=\  \Oh(\exp \{-c_1 z^2 \} ) \text{\ as } z\to -\infty
	 \text{ and } \\
 \label{LogSLNRightTail}
	f_Z(z) \ &=\ \Oh(\exp\{-c_2 z^2 \}) \text{\ as } z \to +\infty
\end{align}
where the constants are as in Lemma~\ref{le:SLNTails}.
\end{corollary}

%!TEX root = ../main.tex
\subsection{Lognormal sums via a normal reference distribution} \label{SS:LNNormalNu}

Consider transforming $S$ to $Z=\ln(S)$ and expanding this density with orthogonal polynomials using a normal distribution as reference. That is, our approximation to $f$ using a $\Norm(\mu, \sigma^2)$ reference is
\[ \widehat{f}_{\Norm} = \frac{1}{\sigma x} \widehat{f}_{Z}\left(\frac{\ln x - \mu}{\sigma} \right) \quad\text{where}\quad \widehat{f}_{Z}(z) = \phi(z) \sum_{i=1}^K a_i \, Q_i(z) \,, \]
with $\phi(\cdot)$ being the standard normal p.d.f. The following result tells us when the integrability condition $f_{Z}/f_\nu \in\mathcal{L}^{2}(\nu)$ is satisfied.
It follows immediately by combining \eqref{SA7.11a} and Corollary~\ref{co:LogSLNTails}

\begin{proposition}\label{eq:IntegrabilityConditionNormalHermiteExpansion}
Consider $Z=\ln(S)$ where $S$ is $\mathcal{SLN}(\boldsymbol{\mu},\bfSigma)$ distributed. Let $\nu$ be the probability measure associated to the normal distribution $\mathcal{N}(\mu,\sigma^2)$. We have $f_Z / f_\nu \in\mathcal{L}^{2}(\nu)$ if
\begin{equation} \label{normal_int_cond_1}
	2 \sigma^2 > (2 c_2)^{-1} = \max_{i=1,\dots,n} \Sigma_{ii} \,.
\end{equation}
\end{proposition}

Computing the $\{ \hat{a}_k \}_{k \in \NZ}$ coefficients can be done using Crude Monte Carlo (CMC), as in
\[ \widehat{a}_k = \frac1R \sum_{r=1}^R Q_n(S_r) \,, \quad S_1, \dots, S_R \iidDist \SLN(\bfmu, \bfSigma) \]
for $k = 0, \dots, K$. We can use the same $S_1$, \dots, $S_R$ for all $\widehat{a}_k$ together with a smoothing technique called \emph{common random numbers} \cite{asmussen2007stochastic,glasserman2003monte}. Note that a non-trivial amount of computational time is typically spent just constructing the Hermite polynomials. Incorporating the Hermite polynomial's recurrence relation in our calculations achieved a roughly $40\times$ speed-up compared with using Mathematica's \texttt{HermiteH}.

%!TEX root = ../main.tex
\subsection{Lognormal sums via a gamma reference distribution} \label{SS:LNGammaNu}

When $\nu$ is Gamma$(r,m)$, it makes little sense to expand $f$ in terms of $\{Q_{k}\}_{k\in\NZ}$ and $f_\nu$ as the integrability condition \eqref{SA7.11b} fails, $f/f_\nu \not\in\L^{2}(\nu)$. The workaround consists in using orthogonal polynomials to expand the \emph{exponentially tilted} distribution, denoted $\SLN_\theta(\bfmu, \bfSigma)$. This distribution's p.d.f.\ is
\begin{equation}\label{ExponentiallyTiltedPDF}
f_\theta(x) = \frac{\e^{-\theta x}f(x)}{\Laplace(\theta)} \,,\quad \theta \ge 0,
\end{equation}
where $\Laplace(\theta)=\Exp[\e^{-\theta S}]$ is the Laplace transform of $S$. Asmussen et al.\ \cite{AsJeRo14b} investigated the use of $f_\theta(x)$ in approximating the survival function of $S$, and developed asymptotic forms and Monte Carlo estimators of this density.
\begin{remark}
The use of gamma distribution and Laguerre polynomials links our approach to a well established technique called the \emph{Laguerre method}. The expansion is an orthogonal projection onto the basis of Laguerre functions constructed by multiplying Laguerre polynomials and the square root of the exponential distribution with parameter $1$. The method is described in \cite{AbChWh95}. Note also that the damping procedure employed when integrability problems arise is quite similar to considering the exponentially tilted distribution instead of the real one. The use of the gamma distribution as reference is applied to actuarial science in \cite{GoLoPo15a,GoLoPo15b}.
$\hfill \Diamond$
\end{remark}
Using \eqref{SA7.11b}, we immediately obtain the following result which
sheds light on how to tune the parameters of the reference gamma distribution so the integrability condition $f_\theta / f_\nu \in\L^2(\nu)$ is satisfied.
\begin{proposition}\label{pr:IntegrabiltyConditionWRTGammaDistribution}
Consider the r.v.\ $S_\theta$ distributed by the exponentially-tilted $\mathcal{SLN}_\theta(\boldsymbol{\mu},\bfSigma)$ distribution. Let $\nu$ be the probability measure associated with the Gamma$(r,m)$ distribution. We have $f_\theta/f_\nu \in\L^2(\nu)$ if
$m>1/2\theta$.
\end{proposition}

Hereafter we assume that the parameters $r$ and $m$ of $f_\nu \sim \mathrm{Gamma}(r,m)$ are chosen to satisfy Proposition~\ref{pr:IntegrabiltyConditionWRTGammaDistribution}'s conditions.

%\subsubsection{Computing the coefficients of the expansion $\{a_{k}\}_{k\in\NZ}$} \label{sss:ComputationCoefficientGammaLaguerreExpansion}
Our approximation---based upon rearranging \eqref{ExponentiallyTiltedPDF}---is of the form
\begin{equation}\label{fGammaApprox}
 \widehat{f}(x) = \e^{\theta x} \Laplace(\theta) \widehat{f}_\theta(x) = \e^{\theta x} \Laplace(\theta) \sum_{k=0}^K a_k Q_k(x) f_\nu(x) \,.
\end{equation}
The coefficients $a_k = \Exp [Q_k(S_\theta)]$ can be estimated in (at least) three different ways: (i) using CMC, (ii) using Monte Carlo with a change of measure so $\theta \to 0$, or (iii) by directly computing the moments $\Exp[S_\theta^k]$. The first method is nontrivial, as simulating from $f_\theta$ likely requires using acceptance-rejection (as in \cite{AsJeRo14b}). Options (ii) and (iii) use
\begin{equation} \label{eq:expand_coeffs}
	a_k = \Exp [Q_k(S_\theta)] =: q_{k0} + q_{k1} \Exp[S_\theta] + \dots + q_{kk} \Exp[S_\theta^k]
\end{equation}
where $\{q_{ki}\}$ are the coefficients in $Q_k$, and
\[
\Exp[S_\theta^i] = \frac{\Exp[S^i \e^{-\theta S}]}{\Laplace(\theta)} =: \frac{\Laplace_i(\theta)}{\Laplace(\theta)} \,.
\]
The $\Laplace_i(\theta)$ notation was selected to highlight the link between $\Exp[S_n^i \e^{-\theta S_n}]$ and the $i$th derivative of $\Laplace(\theta)$.

All three methods require access to the Laplace transform, and method (iii) requires $\Laplace_i(\theta)$, however none of $\Laplace(\theta)$ or $\Laplace_i(\theta)$ are available in closed form. Our approach to circumvent these problems is presented in the Appendix.

%!TEX root = ../main.tex
\section{Numerical illustrations}\label{S:numerical}

We take several approximations $\widehat{f}$ and compare them against the benchmark of numerical integration.
One form of $f$ particularly useful for numerical integration, in terms of the $\LN(\bfmu, \bfSigma)$ density $f_{\mathcal{LN}}$, is as a surface integral,
$ f(s) = n^{-\frac12} \int_{\Delta_n^s} f_{\mathcal{LN}}(\bfx) \dif \bfx $,
where $\Delta_n^s = \{ \bfx \in \RL_+^n : ||\bfx||_1 = s \}$. Mathematica integrates this within a reasonable time for $n=2$ to $4$ using \texttt{NIntegrate} and \texttt{ParametricRegion}). For $n > 4$ we qualitatively assess the performance of the estimators by plotting them.

The quantitative error measure used is the $\L^2$ norm of $(\widehat{f}-f)$ restricted to $(0, \Exp[S])$. We focus on this region as
at one hand it is the hardest to approximate (indeed, Lemma~\ref{le:SLNTails} shows that just a single lognormal is a theoretically justified approximation of the $\SLN$ right tail) and that at the other of high relevance in applications, see for example the introduction of \cite {AsJeRo14b} and the references therein.

\subsection{The estimators}
We will compare the following approximations:
\begin{itemize}
\item the Fenton-Wilkinson approximation $\widehat{f}_{\mathrm{FW}}$, cf.\ \cite{Fe60}, consists in approximating the distribution of $S$ by a single lognormal
with the same first and second moment;
\item the log skew normal approximation $\widehat{f}_{\mathrm{Sk}}$,
cf.\ \cite{hcine2015highly}\footnote{Note that in \cite{hcine2015highly}, the formula for $\epsilon_{\mathrm{opt}}$ contains an typographic error.}, is a refinement of Fenton--Wilkinson by using a log  skew normal as approximation and fitting the left tail in addition to the first and second moment;
% \item $\widehat{f}_{\mathrm{KS}}$ : the kernel density approximation of $\bfS$ (in Mathematica we use the \verb+SmoothKernelDistribution+ with the option ``Bounded'' set so the distribution's support is $(0, \infty)$),
\item the conditional Monte Carlo approximation $\widehat{f}_{\mathrm{Cond}}$ , cf.\ Example~4.3 on p.\ 146 of \cite{asmussen2007stochastic}, uses the representation $f(x)=$
$\Exp\bigl[\mathbb{P}(S\in\dif x\,|\,Y)\bigr]$ for some suitable $Y$ (here chosen as one of the
normal r.v.s $X_i$ occurring in \eqref{7.12a}) and simulates the conditional expectation;
% \item $f_{\Laplace}$ : the inverted Laplace transform approximation of \cite{La15},
\item $\widehat{f}_{\Norm}$ is the approximation described in Section~\ref{SS:LNNormalNu} using a logarithmic transformation
and the Hermite polynomials with a normal reference distribution;
\item $\widehat{f}_{\,\Gamma}$ is the approximation described in Section~\ref{SS:LNGammaNu} using exponential tilting and the Laguerre polynomials with a gamma reference distribution.
\end{itemize}

These approximations are all estimators of functions (i.e., not pointwise estimators, such as in \cite{La15}) and they do not take excessive computational effort to construct. The first two, $\widehat{f}_{\mathrm{FW}}$ and $\widehat{f}_{\mathrm{Sk}}$, only need $\bfmu$ and $\bfSigma$ and do not have any Monte Carlo element. Similarly, the estimator $\widehat{f}_{\,\Gamma}$ when utilising the Gauss--Hermite quadrature described in \eqref{eq:gauss-hermite-quad}
in the Appendix does not use Monte Carlo. For the remaining approximations we utilise the \emph{common random numbers} technique, meaning that the same $R=10^5$ i.i.d.\ $\SLN(\bfmu, \bfSigma)$ samples $\bfS = (S_1, \dots, S_R)^\tr$ are given to each algorithm. Lastly, all the estimators except $\widehat{f}_{\,\Gamma}$ satisfy $\int \widehat{f}(x) \dif x = 1$. One problem with the orthogonal polynomial estimators is that they can take negative values; this can easily be fixed, but we do not make that adjustment here.

% One can instead use the density proportional to $\max\{\widehat{f}(x), 0\}$ with the appropriate normalisation, though we do not  here.

% \subsection{Parameter selection}

% There are many parameters to set in constructing $\widehat{f}_{\Norm}$ and $\widehat{f}_{\,\Gamma}$, including $\mu$, $\sigma^2$, $\theta$, $m$, and $r$.
For $\widehat{f}_{\Norm}$, we take $\mu = \Exp[Z]$ and $\sigma^2 = \Var[Z]$, calculated using numerical integration.
The $\widehat{f}_{\,\Gamma}$ case is more difficult. Equation \eqref{fGammaApprox} shows that we must impose $\theta m < 1$ to ensure that $\widehat{f}_{\,\Gamma}(x)\to 0$ as $x\to \infty$.
Exploring different parameter selections showed that fixing $\theta = 1$ worked reasonably well. Moment matching $f_\theta$ to $f_\nu$ leads to the selection of $m$ and $r$. The moments of $f_\theta$,
$ \widehat{\Exp f_\theta} = \widehat{\Laplace}_1(\theta) / \widehat{\Laplace}_0(\theta)$  and
 % \quad \text{ and } \quad
 $\widehat{\Var f_\theta} = \widehat{\Laplace}_2(\theta) / \widehat{\Laplace}_0(\theta) - \widehat{\Exp f_\theta}^2  $
can be approximated using the Gauss--Hermite quadrature of \eqref{eq:gauss-hermite-quad}; for this we use $H = 64$, 32, 16 for $n=2$, 3, 4 respectively (and CMC for $n > 4$).

With these regimes, parameter selection for the reference distributions is automatic, and the only choice the user must make is in selecting $K$. In these tests we examined various $K$ from 1 to 40, and show the best approximations found. The source code for these tests is available online at \cite{Code}, and we invite readers to experiment the effect of modifying $K$ and $\theta$ and the parameters of the reference distributions.

% TODO (PJL): Mention the $\Laplace^{-1}$ of \cite{La15} approach, but don't actually compute estimates for it. \\

% \enlargethispage{2.75em}

\subsection{Results}

For each test case with $n \le 4$ we plot the $\widehat{f}(x)$ and $f(x)$ together and then $(\widehat{f}(x)-f(x))$ over $x \in (0, 2\Exp[S])$. A table then shows the $\L^2$ errors over $(0, \Exp[S])$.
\setlength\extrarowheight{3pt}

\begin{figure}[H]
\centering
\includegraphics[width=\textwidth]{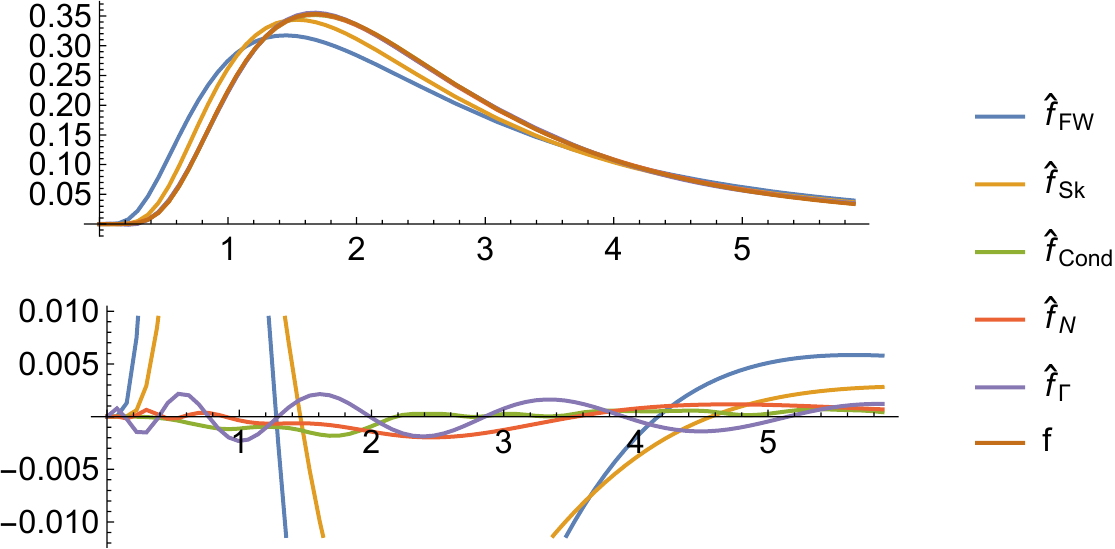}

\vspace{4mm}

\begin{tabular}{c|c|c|c|c|c|}
\cline{2-6}
                         & $\widehat{f}_{\mathrm{FW}}$  & $\widehat{f}_{\mathrm{Sk}}$ & $\widehat{f}_{\mathrm{Cond}}$ & $\widehat{f}_{\Norm}$  & $\widehat{f}_{\,\Gamma}$ \\ \hline
\multicolumn{1}{|c|}{$\L^2$} & $8.01 \stimes 10^{-2}$ & $4.00 \stimes 10^{-2}$ & $1.56 \stimes 10^{-3}$ & $1.94 \stimes 10^{-3}$ & $2.28 \stimes 10^{-3}$ \\  \hline
\end{tabular}
\caption*{Test 1: $\bfmu = (0, 0)$, $\diag(\bfSigma) = (0.5, 1)$, $\rho = -0.2$. Reference distributions used are $\Norm(0.88,0.71^2)$ and Gamma($2.43,0.51$) with $K =$ 32, 16 resp.}
\end{figure}

% \begin{table}[H]
% \centering

% \caption*{Test 1: $\L^2$ errors}
% \end{table}

\begin{figure}[H]
\centering
\includegraphics[width=\textwidth]{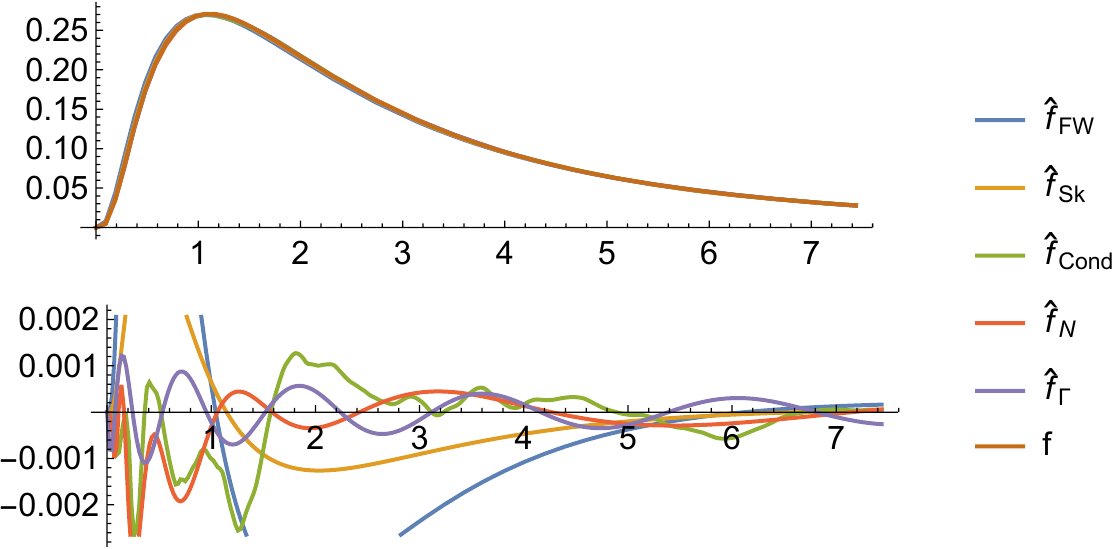}

\vspace{4mm}

\begin{tabular}{c|c|c|c|c|c|}
\cline{2-6}
                         & $\widehat{f}_{\mathrm{FW}}$  & $\widehat{f}_{\mathrm{Sk}}$ & $\widehat{f}_{\mathrm{Cond}}$ & $\widehat{f}_{\Norm}$  & $\widehat{f}_{\,\Gamma}$ \\ \hline
\multicolumn{1}{|c|}{$\L^2$} & $1.02 \stimes 10^{-2}$ & $3.49 \stimes 10^{-3}$ & $1.78 \stimes 10^{-3}$ & $7.86 \stimes 10^{-4}$ & $7.24 \stimes 10^{-4}$ \\ \hline
\end{tabular}
\caption*{Test 2: $\bfmu = (-0.5, 0.5)$, $\diag(\bfSigma)= (1,1)$, $\rho = 0.5$. Reference distributions used are $\Norm(0.91,0.90^2)$ and Gamma($2.35,0.51$) with $K =$ 32, 16 resp.}
\end{figure}

% \begin{table}[H]
% \centering

% \caption*{Test 2: $\L^2$ errors}
% \end{table}

\begin{figure}[H]
\centering
\includegraphics[width=\textwidth]{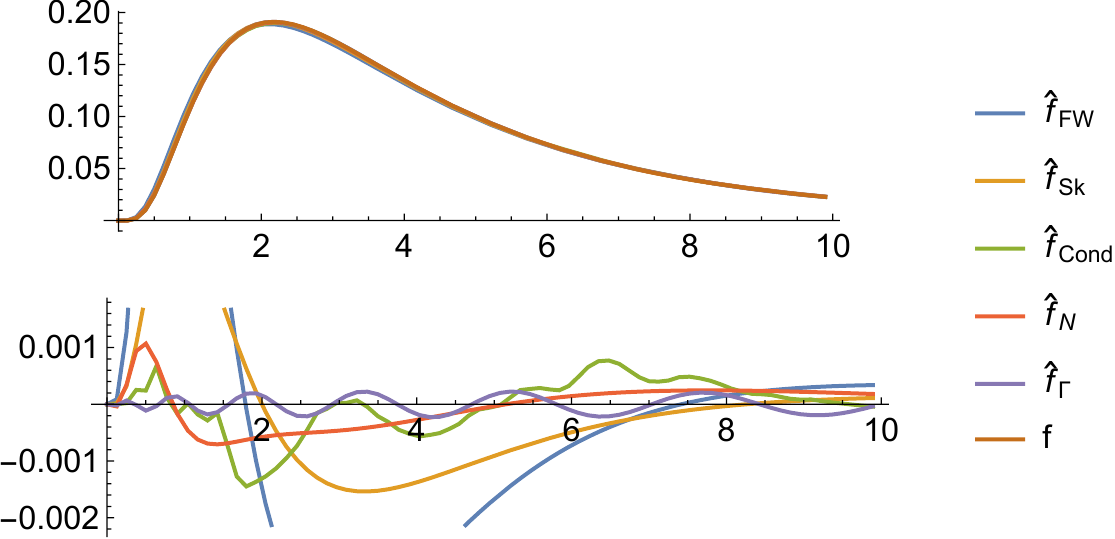}

\vspace{4mm}

\begin{tabular}{c|c|c|c|c|c|}
\cline{2-6}
                         & $\widehat{f}_{\mathrm{FW}}$  & $\widehat{f}_{\mathrm{Sk}}$ & $\widehat{f}_{\mathrm{Cond}}$ & $\widehat{f}_{\Norm}$  & $\widehat{f}_{\,\Gamma}$ \\ \hline
\multicolumn{1}{|c|}{$\L^2$} & $9.48 \stimes 10^{-3}$ & $3.71 \stimes 10^{-3}$ & $1.60 \stimes 10^{-3}$ & $1.18 \stimes 10^{-3}$ & $3.53 \stimes 10^{-4}$ \\ \hline
\end{tabular}
\caption*{Test 3: $n=3$, $\mu_i = 0$, $\Sigma_{ii} = 1$, $\rho = 0.25$. Reference distributions used are $\Norm(1.32,0.74^2)$ and Gamma($3,0.57$) with $K =$ 7, 25 resp.}
\end{figure}

% \begin{table}[H]
% \centering

% \caption*{Test 3: $\L^2$ errors}
% \end{table}

\begin{figure}[H]
\centering
\includegraphics[width=\textwidth]{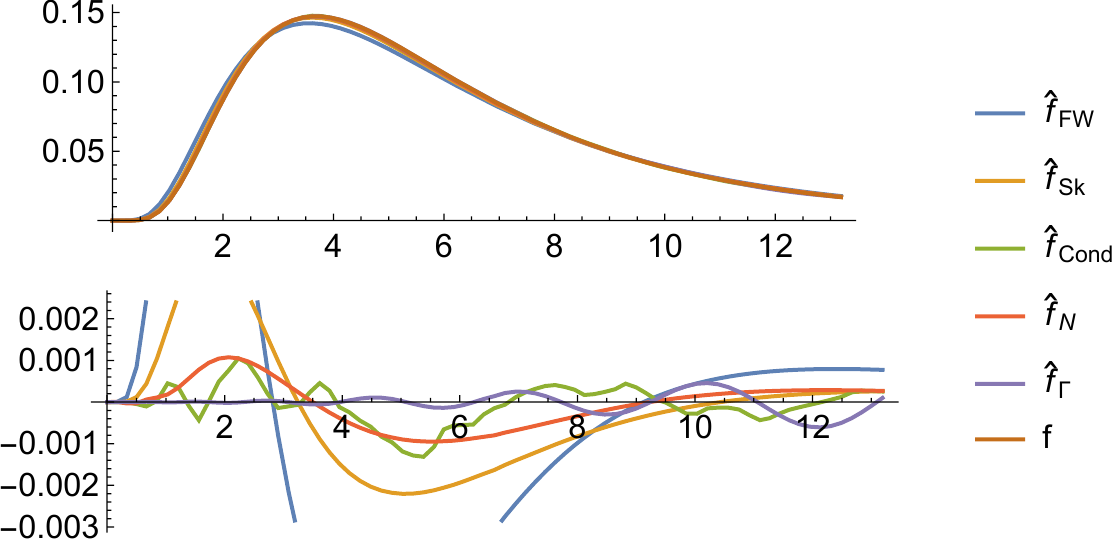}

\vspace{4mm}

\begin{tabular}{c|c|c|c|c|c|}
\cline{2-6}
                         & $\widehat{f}_{\mathrm{FW}}$  & $\widehat{f}_{\mathrm{Sk}}$ & $\widehat{f}_{\mathrm{Cond}}$ & $\widehat{f}_{\Norm}$  & $\widehat{f}_{\,\Gamma}$ \\ \hline
\multicolumn{1}{|c|}{$\L^2$} & $1.82 \stimes 10^{-2}$ & $6.60 \stimes 10^{-3}$ & $1.90 \stimes 10^{-3}$ & $1.80 \stimes 10^{-3}$ & $1.77 \stimes 10^{-4}$ \\ \hline
\end{tabular}
\caption*{Test 4: $n=4$, $\mu_i = 0$, $\Sigma_{ii}=1$, $\rho = 0.1$. Reference distributions used are $\Norm(1.32,0.74^2)$ and Gamma($3.37,0.51$) with $K =$ 18, 18 resp.}
\end{figure}

% \begin{table}[H]
% \centering

% \caption*{Test 4: $\L^2$ errors}
% \end{table}

The following test case shows the density approximations for a large $n$.

% The following two test cases show the various $\widehat{f}$ in cases where $n > 4$.

\begin{figure}[H]
\centering
\includegraphics[width=\textwidth]{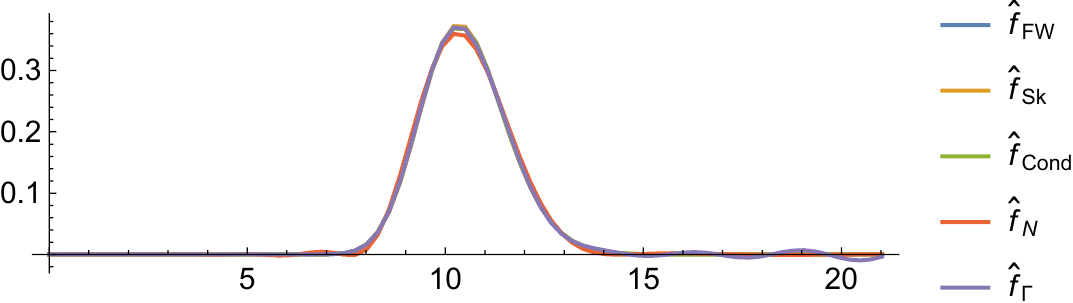}
\caption*{Test 5: Sum of 10 i.i.d.\ $\LN(0,0.1)$ r.v.s. Reference distributions used are $\Norm(2.35,0.23^2)$ and Gamma($12.61,0.25$) with $K =$ 18, 35 resp.}
\end{figure}

% \begin{figure}[H]
% \centering
% \includegraphics[width=\textwidth]{d6.pdf}
% \caption*{Test 6: Sum of 20 i.i.d.\ $\LN(-2,1)$ r.v.s. Reference distributions used are $\Norm(1.46,0.71^2)$ and Gamma($8.78,0.25$) with $K =$ 40, 10 resp.}
% \end{figure}

Finally, we fit $\widehat{f}_{\Norm}$ and $\widehat{f}_{\,\Gamma}$ to simulated data ($10^5$ replications) for the sum of lognormals with a non-Gaussian dependence structure. Specifically, we take the sum of $n = 3$ standard lognormal r.v.s with a \emph{Clayton copula}, defined by its distribution function
\[
C^{\text{Cl}}_\theta(u_1, \dots, u_n) = \Big( 1 - n + \sum_{i=1}^n u_i^{-\theta} \Big)^{-1/\theta}, \qquad \text{ for } \theta > 0 \,.
\]
The Kendall's tau correlation of the $C^{\text{Cl}}_\theta$ copula is $\tau = \theta / (\theta + 2)$ \cite{mcneil2015quantitative}.

\begin{figure}
\centering
\includegraphics[width=\textwidth]{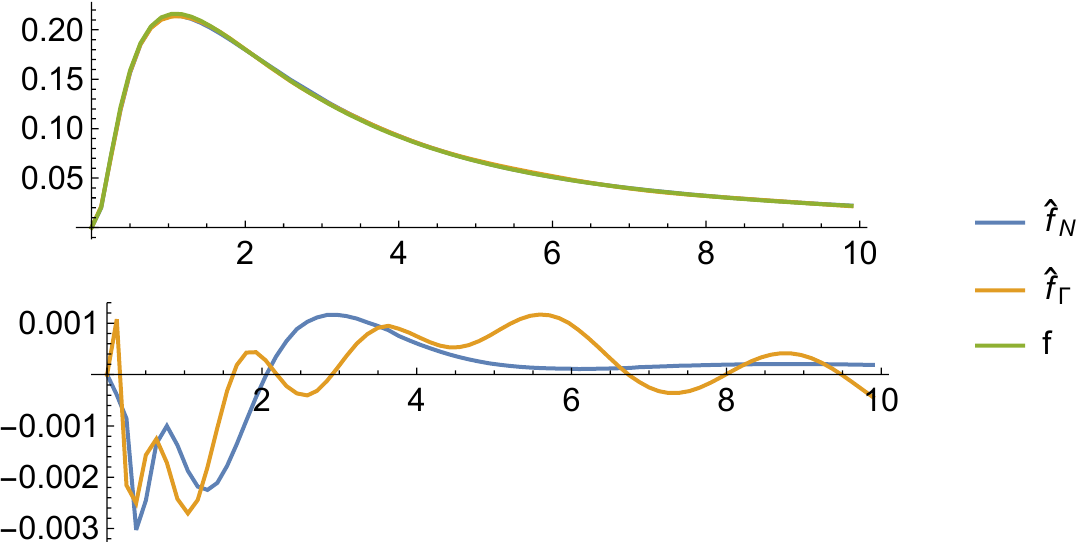}
\caption*{Test 6: Sum of 3 $\LN(0,1)$ r.v.s with $C^{\text{Cl}}_{10}(\cdot)$ copula (i.e., $\tau = \frac56$). Reference distributions used are $\Norm(1.46,0.71^2)$ and Gamma($8.78,0.25$) with $K = 40$. The $\L^2$ errors of $\widehat{f}_{\Norm}$ and $\widehat{f}_{\,\Gamma}$ are $2.45 \times 10^{-3}$ and $2.04 \times 10^{-3}$ respectively.}
\end{figure}

% \newpage

% Some preliminary results. Comparing sum of $n = 2, 3, 4$ independent standard lognormals. The Monte Carlo methods have $R=10^3$ replications (this will be increased). The two orthogonal polynomial expansions --- $f_{\Norm}$ and $f_{\,\Gamma}$ --- have $K=30$ terms, and $f_{\,\Gamma}$ uses Gauss--Hermite integration in \eqref{eq:gauss-hermite-quad} to get the $\widehat{\Laplace}_i(\theta)$ with order $H = 64, 32, 16$ respectively. The parameter for the $f_{\,\Gamma}$ reference distribution is $r = m = \theta = 1$, i.e., approximating $f_{\theta=1}$ with a $x\e^{-x}$ density. For $f_{\Norm}$, it has mean $\mu = \widehat{\Exp[\log S]}$ and variance $\sigma^2 = \widehat{\Var[\log S]}$, found by CMC (still cannot find a better solution than this). \\

% The plots show $s \in [0, \Exp[S]]$, and the errors are measured as $\L^2(\widehat{f})=\int_{[0,\Exp[S]]} |\widehat{f}-f|^2 \dif s$. Also listed is the vector of $\L^2$ errors scaled so the minimum value $= 1$ for easier comparison of methods.

%!TEX root = ../main.tex
%\section{Conclusion}

Our overall conclusion of the numerical examples is that no single method can
be considered as universally superior. Of the methods in the literature,
the log skew normal approximations is generally better than Fenton-Wilkinson,
which is unsurprising given it is an extension introducing one more parameter.
The estimators, $\widehat{f}_{\Norm}$ and $\widehat{f}_{\,\Gamma}$, based on orthogonal polynomial approximation techniques, are very flexible. They also display as least as good and sometimes better p.d.f.\ estimates over the interval $(0, \Exp[S])$ and their periodic error indicates that they would supply even more accurate c.d.f.\ estimates. One should note, however, that their
performance relies on the tuning of parameters and that somewhat
greater effort is involved in their computation (though this is mitigated through the availability
of the software in \cite{Code}).

An interesting feature of $\widehat{f}_{\Norm}$ and $\widehat{f}_{\,\Gamma}$ is that
the Frank copula example indicates some robustness to the dependence structure used.
In view of the current interest in financial applications of non-Gaussian dependence this
seems a promising line for future research.
%!TEX root = ../main.tex
\subsection*{Acknowledgements}
{\small We are grateful to Jakob Schach M\o ller for helpful discussions on orthonormal expansions.}

\begin{appendix}

\renewcommand{\theequation}{A.\arabic{equation}}
\setcounter{proposition}{0}
\renewcommand{\theproposition}{A.\arabic{proposition}}
\setcounter{remark}{0}
\renewcommand{\theremark}{A.\arabic{remark}}

\section*{A.1 Proof of Proposition \ref{LogNormalOrthogonalPolynomialProposition}} \label{app:prop_proof}
\begin{proof}
The polynomials orthogonal with respect to the lognormal distribution will be derived using the general formula
\begin{equation}\label{eq:OrthogonalPolynomialGeneralExpression}
Q_{n}(x)=\frac{1}{\sqrt{D_{n-1,n-1}D_{n,n}}}\displaystyle
\left| \begin{array}{cccc}
s_{0}&s_{1}&\cdots&s_{n} \\
s_{1}&s_{2}&\cdots&s_{n+1}  \\
\vdots&\vdots&&\vdots\\
s_{n-1}&s_{n}&\cdots&s_{2n-1}\\
1&x&\cdots&x^{n}
\end{array}
\right|,\hspace{0.1cm}n\geq1,
\end{equation}
where $\{s_{n}\}_{n\in\NZ}$ denotes the moment sequence of the lognormal distribution and  $D_{n,n}=\left|\left\{s_{k,l}\right\}_{0\leq k,l \leq n}\right|$ is a Hankel determinant. The moments of the lognormal distribution are given by $s_{n}=p^{n}q^{n^{2}}$, where $p=\e^{\mu}$ and $q=\e^{\frac{\sigma^{2}}{2}}$. Consider
\begin{equation}\label{eq:Dn1}
D_{n,n}=\displaystyle
\left| \begin{array}{cccc}
1&pq&\cdots&p^{n}q^{n^{2}} \\
pq&p^{2}q^{4}&&p^{n+1}q^{(n+1)^{2}}\\
\vdots&\vdots&&\vdots\\
p^{n-1}q^{(n-1)^{2}}&p^{n}q^{n^{2}}&\cdots&p^{2n-1}q^{(2n-1)^{2}}\\
p^{n}q^{n^{2}}&p^{n+1}q^{(n+1)^{2}}&\cdots&p^{2n}q^{(2n)^{2}}
\end{array}
\right|,\hspace{0.1cm}n\geq1,
\end{equation}
and denote by $R_{k}$ the $k$th row and by $C_{\ell}$ the $\ell$th column. We apply the elementary operations $R_{k+1}\rightarrow p^{-k}q^{-k^{2}}R_{k+1}$, and $C_{\ell+1}\rightarrow p^{-\ell}q^{-\ell^{2}}C_{\ell+1}$ for $k,\ell=0,\ldots,n$ to get a Vandermonde type determinant. Thus we have
\begin{equation}\label{eq:Dn2}
D_{n,n}=\e^{n(n+1)\mu}\e^{\frac{n(n+1)(2n+1)}{3}\sigma^{2}}\prod_{k=0}^{n-1}
\left[\e^{-\sigma^{2}};\e^{-\sigma^{2}}\right]_{k}\\
\end{equation}
We expand the determinant in \eqref{eq:OrthogonalPolynomialGeneralExpression} with respect to the last row to get
\begin{equation}\label{eq:PolynomialExpression1}
Q_{n}(x)=\frac{1}{\sqrt{D_{n-1,n-1}D_{n,n}}}\sum_{k=0}^{n}(-1)^{n+k}x^{k}D_{n-1,n}^{-k},
\end{equation}
where $D_{n-1,n}^{-k}$ is $D_{n,n}$ with the last row and the $(k+1)$th column deleted. We perform on  $D_{n-1,n}^{-k}$ the following operations: $R_{j+1}\rightarrow p^{-j}q^{-j^{2}}R_{j+1}$, for $j=0,\ldots,n-1$, $C_{j+1}\rightarrow p^{-j}q^{-j^{2}}C_{j+1}$, for $j=0,\ldots,k-1$, and finally $C_{j}\rightarrow p^{-j}q^{-j^{2}}C_{j}$, for $j=k+1,\ldots,n$. We obtain
\begin{equation*}
D_{n-1,n}^{-k}=p^{n^{2}-k}q^{\frac{2n^{3}+n}{3}-k^{2}}\left| \begin{array}{ccccccc}
1&\alpha_{0}&\cdots&\alpha_{0}^{k-1}&\alpha_{0}^{k+1}&\cdots&\alpha_{0}^{n} \\
1&\alpha_{1}&\cdots&\alpha_{1}^{k-1}&\alpha_{1}^{k+1}&\cdots&\alpha_{1}^{n+1}  \\
\vdots&\vdots&&\vdots&\vdots&&\vdots\\
1&\alpha_{n-1}&\cdots&\cdots&\cdots&\cdots&\alpha_{n-1}^{n}
\end{array}
\right|,
\end{equation*}
where $\alpha_{k}=q^{2k}$, for $k=0,\ldots,n-1$. Expanding the polynomial $B(X)=\prod_{i=0}^{n-1}(x-\alpha_{i})$, we get
\begin{equation*}
B(x)=x^{n}+\beta_{n-1}x^{n-1}+\ldots+\beta_{0},
\end{equation*}
where $\beta_{k}=(-1)^{n-k}e_{n-k}\left(\alpha_{0},\ldots,\alpha_{n-1}\right)$, and $e_{k}\left(X_{1},\ldots,X_{n}\right)$ denotes the elementary symmetric polynomial, defined previously in  \eqref{eq:ElementarySymmetricPolynomial}. We apply the elementary operation $C_{n}\rightarrow C_{n}+\sum_{j=0}^{k-1}a_{j}C_{j+1}+\sum_{j=k+1}^{n-1}a_{j}C_{j}$, followed by $n-k$ cyclic permutations to get
\begin{equation}\label{eq:MinorDnn}
D_{n-1,n}^{-k}=p^{n-k}q^{n^{2}-k^{2}}e_{n-k}\left(1,\ldots,q^{2(n-1)}\right)D_{n-1,n-1}.
\end{equation}
Inserting \eqref{eq:Dn2} and \eqref{eq:MinorDnn} into \eqref{eq:PolynomialExpression1} leads to \eqref{eq:OrthonormalPolynomialLogNormalExpression}.
\end{proof}

\section*{A.2 Computing the coefficients of the expansion $\{a_{k}\}_{k\in\NZ}$
in the gamma case} \label{app:proof}

We extend here the techniques developed in \cite{La15} to construct an approximation
for $\Laplace_i(\theta)$. We note that $\Laplace_i(\theta) \propto \int_{\RL^n} \exp\{ -h_{\theta,i}(\bfx) \} \dif \bfx $ where
\[ h_{\theta,i}(\bfx) = - i \ln(\bfone^\tr \ve^{\bfmu + \bfx}) + \theta \bfone^\tr \ve^{\bfmu + \bfx} + \frac12 \bfx^\tr \bfSigma^{-1} \bfx \,, \quad i \in \NZ \,. \]
This uses the notation $\ve^{\bfx} = (\e^{x_1}, \dots, \e^{x_n})^\top$. Next, define $\bfx^*$ as the minimiser of $h_{\theta,i}$ (calculated numerically), and consider a second order Taylor expansion of $h_{\theta,i}$ about $\bfx^*$. Denote $\widetilde{\Laplace}_i(\theta)$ as the approximation where $h_{\theta,i}$ is replaced by this Taylor expansion in $\Laplace_i(\theta)$. Simplifying yields
\begin{equation} \label{firstTerm}
	\widetilde{\Laplace}_i(\theta) = \frac{\exp\{-h_{\theta,i}(\bfx^*)\}}{\sqrt{|\bfSigma \bfH|}}
\end{equation}
where $\bfH$, the Hessian of $h_{\theta,i}$ evaluated at $\bfx^*$, is
\[ \bfH = i\frac{\ve^{\bfmu + \bfx^*} (\ve^{\bfmu + \bfx^*})^\tr}{(\bfone^\tr \ve^{\bfmu + \bfx^*})^2} + \bfSigma^{-1} - \diag(\bfSigma^{-1} \bfx^*) \,. \]

As $\theta \to \infty$ we have $\widetilde{\Laplace}_i(\theta) \to \Laplace_i(\theta)$. We can rewrite $\Laplace_i(\theta) = \widetilde{\Laplace}_i(\theta) I_i(\theta)$ and estimate $I_i(\theta)$, as in \cite{La15}.

\begin{proposition} \label{prop:laplace_derivatives}
The moments of the exponentially-tilted distribution $\SLN_\theta(\bfmu, \bfSigma)$, denoted $\Laplace_i(\theta)$, can be written as $\Laplace_i(\theta) = \widetilde{\Laplace}_i(\theta) I_i(\theta)$ where $\widetilde{\Laplace}_i(\theta)$ is in \eqref{firstTerm} and
\begin{align*}
	I_i(\theta) &= \sqrt{|\bfSigma \bfH|} \, v(\bfzero)^{-1} \, \Exp[v(\bfSigma^{\frac12} Z)]
\end{align*}
where $Z \sim \Norm(\bfzero, \bfI)$, and
\[ v(\bfz) = \exp\{ i \ln(\bfone^\tr \ve^{\bfmu + \bfx^* + \bfz}) - \theta \bfone^\tr \ve^{\bfmu + \bfx^* + \bfz} - (\bfx^*)^\tr \bfSigma^{-1} \bfz \} \,. \]
\end{proposition}
\begin{proof}
We begin by substituting $\bfx = \bfx^* + \bfH^{-\frac12} \bfy$ into $\Laplace_i(\theta)$, then multiply by $\exp\{ \pm \text{ some constants}\,\}$:
\begin{align*}
	\Laplace_i(\theta) &= \int_{\RL^n} \frac{ (2\pi)^{-\frac{n}{2}} }{\sqrt{|\bfSigma|}} \exp\{ i \log(\bfone^\tr \ve^{\bfmu + \bfx}) - \theta \bfone^\tr \ve^{\bfmu + \bfx} - \frac12 \bfx^\tr \bfSigma^{-1} \bfx \}  \dif \bfx \\
	&= \int_{\RL^n} \frac{ (2\pi)^{-\frac{n}{2}} }{\sqrt{|\bfSigma \bfH|}} \exp\{ i \log(\bfone^\tr \ve^{\bfmu + \bfx^* + \bfH^{-\frac12} \bfy}) - \theta \bfone^\tr \ve^{\bfmu + \bfx^* + \bfH^{-\frac12} \bfy} \\
	&\qquad - \frac12 (\bfx^* + \bfH^{-\frac12} \bfy)^\tr \bfSigma^{-1} (\bfx^* + \bfH^{-\frac12} \bfy) \}  \dif \bfy \\
	&= \widetilde{\Laplace}_i(\theta) \exp\{ -i \log(\bfone^\tr \ve^{\bfmu + \bfx^*}) + \theta \bfone^\tr \ve^{\bfmu + \bfx^*} \} \\
	&\qquad \times \int_{\RL^n} (2\pi)^{-\frac{n}{2}} \exp\{ i \log(\bfone^\tr \ve^{\bfmu + \bfx^* + \bfH^{-\frac12} \bfy}) - \theta \bfone^\tr \ve^{\bfmu + \bfx^* + \bfH^{-\frac12} \bfy} \\
	&\qquad\qquad -(\bfx^*)^\tr \bfSigma^{-1} \bfH^{-\frac12} \bfy - \frac12 \bfy^\tr (\bfSigma \bfH)^{-1} \bfy \}  \dif \bfy \,.
\end{align*}
That is, $\Laplace_i(\theta) = \widetilde{\Laplace}_i(\theta) I_i(\theta)$. In $I_i(\theta)$, take the change of variable $\bfy = (\bfSigma\bfH)^{\frac12} \bfz$, and the result follows.
% \begin{align*}
% 	I_i(\theta) &= \sqrt{|\bfSigma \bfH|} \exp\{ -i \log(\bfone^\tr \ve^{\bfmu + \bfx^*}) + \theta \bfone^\tr \ve^{\bfmu + \bfx^*} \} \\
% 	&\qquad \times \int_{\RL^n} (2\pi)^{-\frac{n}{2}} \exp\{ i \log(\bfone^\tr \ve^{\bfmu + \bfx^* + \bfSigma^{\frac12}\bfz}) - \theta \bfone^\tr \ve^{\bfmu + \bfx^* + \bfSigma^{\frac12} \bfz} \\
% 	&\qquad\qquad -(\bfx^*)^\tr \bfSigma^{-1} \bfSigma^{\frac12} \bfz - \frac12 \bfz^\tr \bfI \bfz \} \dif \bfz \\
% 	&= \sqrt{|\bfSigma \bfH|} \, v(\bfzero)^{-1} \, \Exp[v(\bfSigma^{\frac12} Z)]
% \end{align*}
% with $v$ and $Z$ as defined in the proposition.
\end{proof}

\begin{remark}
The form of $I_i(\theta)$ naturally suggests evaluation using \emph{Gauss--Hermite} quadrature:
\begin{equation} \label{eq:gauss-hermite-quad}
	\widehat{\Laplace}_i(\theta) = \frac{\exp\{-h_{\theta,i}(\bfx^*)\}}{v(\bfzero) \,\pi^{\,n/2}} \sum_{i_1=1}^H \cdots \sum_{i_n=1}^H v(\bfSigma^{\frac12} \bfz)  \prod_{j=1}^n w_{i_j}
\end{equation}
where $\bfz = (z_{i_1}, \dots, z_{i_n})^\top$, the set of weights and nodes $\{(w_i, z_i) : 1 \le i \le H\}$ is specified by the Gauss--Hermite quadrature algorithm, and $H \ge 1$ is the order of the approximation. This approximation is accurate, especially so when the $i$ in $\Laplace_i$ becomes large. Even for $\Laplace$ ($= \Laplace_0$) this method appears to outperform the quasi-Monte Carlo scheme outlined in \cite{La15}. $\hfill \Diamond$
\end{remark}

Thus, with $\widehat{\Laplace}_i(\theta)$ given in \eqref{eq:gauss-hermite-quad}, we can now estimate the coefficients. The three methods correspond to
\begin{enumerate}
\item $\widehat{a}_k = R^{-1} \sum_{r=1}^R Q_k(S_r)$, for $S_1$, \dots, $S_R \iidDist f_\theta(x)$,
\item $\widehat{a}_k = \sum_{j=0}^k q_{kj} \, \widehat{\Exp[S_\theta^j]} = q_{k0} + (R \, \widehat{\Laplace}(\theta))^{-1} \sum_{j=1}^k q_{kj} \sum_{r=1}^R S_r^j \e^{-\theta S_r}$, from \eqref{eq:expand_coeffs}, where $S_1$, \dots, $S_R \iidDist f(x)$,
\item $\widehat{a}_k = q_{k0} + \widehat{\Laplace}(\theta)^{-1} \sum_{j=1}^k q_{kj} \, \widehat{\Laplace}_j(\theta)$.
\end{enumerate}
In the numerical illustrations, we switched between using methods (2) and (3) for large and small $n$ respectively. Algorithms for efficient simulation from $f_\theta$ is work in progress.

% \begin{remark}
% In \cite{AsJeRo14b}, the authors derived a simulation algorithm to simulate replications directly from the exponentially tilted distribution of the univariate lognormal distribution. Thus, we can access replications drawn from the exponentially tilted distribution of the sum of independent lognormals.
% \end{remark}

\end{appendix}

\bibliography{main}
\bibliographystyle{ws-book-van}

\end{document}